\documentclass[onenumber,plainnumbering,english,11pt,standalone]{amsart}
\usepackage[utf8]{inputenc}
\usepackage[square,numbers]{natbib}
\usepackage[mathscr]{eucal}
\usepackage{amssymb,amsmath,amscd,mathrsfs,paralist,amsthm}
\usepackage{mathrsfs}
\usepackage{bm}
\usepackage{scalerel}
\usepackage{lipsum}

\usepackage[T1]{fontenc}
\usepackage{tikz}
\usetikzlibrary{shapes.geometric, arrows}                      
\usepackage{tikz}
\usepackage{bm}
\usepackage{graphicx}
\usepackage{pgfplots}
\pgfplotsset{compat=1.15}
\usetikzlibrary{positioning, shapes.geometric}
\usepgfplotslibrary{colormaps}
 
\usepackage{tikz-cd}
\usetikzlibrary{arrows, matrix}
\usepackage[all,cmtip]{xy}
\usepackage[colorlinks,linkcolor=black,urlcolor=black]{hyperref}
\hypersetup{
     colorlinks   = true,
     citecolor    = black}
\usepackage{pgfplots}
\pgfplotsset{compat=newest}
\usepackage[font=small,labelsep=none]{caption}
\usetikzlibrary{matrix}
\usetikzlibrary{cd}
\usepackage[colorlinks,linkcolor=black,urlcolor=black]{hyperref}
\hypersetup{
     colorlinks   = true,
     citecolor    = black}
\usepackage[all,cmtip]{xy}
\theoremstyle{plain}
\newtheorem{theorem}{Theorem}[section]
\newtheorem*{theo}{Theorem}
\newtheorem{lemma}[theorem]{Lemma}

\newtheorem{corollary}[theorem]{Corollary}
\theoremstyle{definition}
\newtheorem{definition}[theorem]{Definition}
\newtheorem{pdefinition}[theorem]{Proposition-Definition}
\newtheorem{example}[theorem]{Example}

\theoremstyle{remark}
\newtheorem{remark}[theorem]{Remark}
\makeatletter
\let\amstexbig\big
\def\newbig#1{%
  \ifx#1|%
    \expandafter\@firstoftwo
  \else
    \expandafter\@secondoftwo
  \fi
  {\big@bar}%
  {\amstexbig{#1}}%
}
\AtBeginDocument{\let\big\newbig}
\def\big@bar{\bBigg@{1.1}|}
\makeatother

\DeclareMathOperator{\image}{Im}

\DeclareMathOperator{\rank}{rank}
\DeclareMathOperator{\identity}{Id}
\DeclareMathOperator{\diag}{diag}
\DeclareMathOperator{\kernel}{Ker}

\numberwithin{equation}{section}
\newcommand{\R}{\mathbb R}

\newcommand{\Z}{\mathbb Z}
\newcommand{\C}{\mathbb C}

\newcommand{\theoref}[1]{Theorem~\ref{#1}}

\newcommand{\lemref}[1]{Lemma~\ref{#1}}

\newcommand{\corref}[1]{Corollary~\ref{#1}}

\newcommand{\defiref}[1]{Definition~\ref{#1}}
\newcommand{\secref}[1]{Section~\ref{#1}}
\begin{document}
\title{The Birkhoff-Grothendieck theorem}
\author{ Oumar Wone}
\address{Oumar Wone, 
Chapman University, One university Drive \\
 Orange, CA 92866, USA.}
\email{wone@chapman.edu}
\begin{abstract}We give a historical presentation of the Grothendieck theorem on the splitting of vector bundles over the Riemann sphere, and explore some of its links with the Riemann-Hilbert-Birkhoff problems and the Birkhoff factorization theorem.
  \end{abstract}
 \maketitle

\section{Introduction}
We give in this article a historical exposition of the Grothendieck splitting theorem for vector bundles on the Riemann sphere. The Grothendieck theorem has a long and distinguished story and has over the centuries been discovered by many mathematicians. The first discoverers of the theorem were in hindsight, maybe Dedekind and Weber, who proved the following \cite[prop.~3.1]{hazewinkel}
\begin{theo}[Dedekind-Weber]
Let $\rm L$ be a field, for $x$ a variable, we consider the ring $\rm L\left[x,x^{-1}\right]$. For a given matrix $A\in\rm{GL}(n,L\left[x,x^{-1}\right])$, there exists matrices $B\in\rm{GL}(n,L\left[x\right])$, $C\in\rm{GL}(n,L\left[x^{-1}\right])$ such that
\[
  BAC =
  \begin{bmatrix}
    x^{d_1} & & &\\
      &x^{d_2} & &\\
    &  &\ddots& \\
    & & &x^{d_n}
  \end{bmatrix}
\]

is a diagonal matrix, and where $d_1\geqslant d_2\geqslant\ldots\geqslant d_n$, $d_i\in\Z$ and the sequence $d_i$ is unique.
\end{theo}
When $\rm L=\C$, the Dedekind-Weber theorem is equivalent to the splitting of algebraic vector bundles on the complex projective line, see \cite{hazewinkel} or equivalently by GAGA to the splitting of holomorphic vector bundles on the Riemann sphere, that is to the theorem of Grothendieck \cite{grothendieck}. 

In the analytic setting a similar factorization of matrices of holomorphic functions were later proved by Hilbert, Plemelj, and Birkhoff, \cite{birkhoff1, birkhoff2, birkhoff3, hilbert3, plemelj, plemelj1}. This factorization of holomorphic matrices is equivalent to the Grothendieck theorem, see \corref{birkhofffac} for one direction of this equivalence. 

Hilbert, Plemelj and Birkhoff were not interested in the classification of vector bundles over the Riemann sphere per say, but in the so-called classical Riemann-Hilbert problem and the Riemann-Hilbert boundary value problem, Hilbert \cite{hilbert1}, Plemelj \cite{plemelj1}, Birkhoff \cite{birkhoff3}. On the other hand the Birkhoff problem was mainly investigated by Birkhoff \cite{birkhoff1}. 

The classical Riemann-Hilbert was problem $21$ on the famous list of Hilbert problems given in 1900, \cite{hilbert1}. It consists in finding a linear homogeneous differential system on a Riemann surface $M$, $\dfrac{dw}{dz}=A(z)w$, with $A(z)\in \rm{M}(n,\C)$, with prescribed singularities set $\mathscr S$ (at each point of which it must have only singularity of the first kind, i.e. simple poles), and with prescribed monodromy representation $\Psi:\pi_1(M\backslash \mathscr S,P)\to\rm{GL}(n,\C)$, $P\in M\backslash \mathscr S$. In this strong formulation the problem has no solution without some additional hypothesis. Indeed Bolibrukh gave counterexamples to it in \cite{boli2, boli1}, see also \cite{kostov}. But if one requires that all the singularities be regular singular (moderate growth of solutions) then the problem has been solved many times since its formulation by Hilbert, and by many mathematicians over the years. Among others we may mention Plemelj \cite{plemelj}, who based his study on the so-called Riemann-Hilbert boundary value problem for singular integral equations, \cite{hilbert2, muskhelishvili, vekua}. Then we have the work of Birkhoff in \cite{birkhoff3}, who using  successive approximations, simplified the work of Plemelj. Following that Garnier \cite{garnier1}, simplified and extended the result of Birkhoff \cite{birkhoff3} on the Riemann-Hilbert boundary value problem, to a more general class of data. Moreover Garnier \cite{garnier2} was interested in the solution of the Riemann-Hilbert problem in family, i.e. the problem of isomonodromy. This means in modern terms that one wants to find the necessary and sufficient conditions for a family of vector bundles with flat connections on a suitable complex manifold, to have "constant" monodromy. We refer to \cite{sabbah} and references cited in there for a nice account on this topic. There is also the interesting work of Lappo-Danilevsky \cite{lapp}, who based his study on the use of analytic functions on matrices, to express the solutions of a Fuchsian system and the monodromy matrices as convergent series of the matrix coefficients of the system. The first modern attempt to the solution of the Riemann-Hilbert problem in the regular singular case was done by Levelt \cite{levelt}. Then Rohrl \cite{rohrl}, using the cohomological framework and the theory of fiber bundles, rediscovered the theorem. Later on Deligne \cite{deligne} generalized the work of Rohlr for vector bundles on varieties of higher dimensions. See also \cite{andre, katz}. Besides there is the Tannakian approach to the Riemann-Hilbert problem, as emphasized by Katz. In order to explain it, one needs firstly to generalize differential systems (equations) on a Riemann surface $M$ to vector bundles on complex manifolds $M$ equipped with flat holomorphic connections. Then one builds a tensor category $\rm{D.E.}(\mathit M)$ of 'differential' equations on $M$. With that in hand Katz has shown that there is an equivalence of tensor categories between the category $\rm{D.E.}(\mathit M)$ of differential equations on $M$, and the tensor category of finite dimensional $\C$-representations of the fundamental group of $M$, with chosen base point. One may consult \cite{katz1, katz3}. In another direction Katz has popularized the interesting Grothendieck-Katz conjecture, which originates from the problem, first studied by Schwartz \cite{schwarz}, of finding when the Gaussian hypergeometric differential equation admits a full set of algebraic solutions over the field of rational functions $\C(x)$. Schwarz \cite{schwarz}, and quite recently Beukers-Heckman, for the generalized hypergeometric differential equation, have shown that this happens precisely when one has a so-called interlacing property, on the unit circle; see \cite{schwarz, beukers}. Using methods from algebraic geometry, and restricting to the arithmetic setting, Katz \cite{katz2, katz4}, recovered the results of Schwartz and Beukers-Heckman. His approach is via the study of the $p$-curvatures of their associated connections. When all but a finite number of their $p$-curvatures vanish, he showed that the connections in question have a full set of algebraic solutions. More generally, Katz showed for the Gauss-Manin connections and their direct factors, that the vanishing of sufficiently many of their $p$-curvatures is equivalent to the connections having a full set of algebraic solutions. The Grothendieck-Katz conjecture is the conjectural equivalence, for a given a connection on a vector bundle on a variety over a number field, of the fact that the vanishing of all but finitely many $p$-curvatures of the connection is equivalent to the statement that the connection has a full set of algebraic solutions. Very recently, André \cite{andre1} and Bost \cite{bost}, independently, made important contributions in the study of the Grothendieck-Katz conjecture. Both used suitable reformulations of the algebraicity criterion of Chudnovsky \cite{chudnovsky1, chudnovsky2, chudnovsky3} which is reminiscent of the classical algebraicity criterion of Borel-Dwork, to show that the Grothendieck-Katz conjecture holds when the differential Galois group of the connection has a solvable neutral component.

Let us briefly explain the Birkhoff problem in a nutshell. It consists in the following: one considers the Riemann sphere $\mathbb P^1(\C)$ together with the two points $\{a_1=\infty, a_2=0\}$. To $a_2=0$ it is assigned the Poincaré rank $r_2=0$, and at $a_1=\infty$ the Poincaré rank $r_1=r$, $r\in\Z_{+}$ and any Stokes data, are fixed. Then one asks for the existence of a linear homogenous differential system on $\mathbb P^1(\C)$ which is singular at only the two points $\{a_1=\infty, a_2=0\}$ and with the aforementioned fixed singularity profile. 

Let us also remark that the Birkhoff and Riemann-Hilbert problems are of a similar kind since they can be subsumed under the so-called Generalized Riemann-Hilbert problem, see \cite{boli3}.

The plan of this article is as follows: we introduce in \secref{vec} the notations necessary in order to formulate the Grothendieck splitting theorem, by giving basics on vector bundles on Riemann surfaces. Then in \secref{birkgrothen} we formulate the theorem of Grothendieck, \theoref{grothendieck} and give some further work by Atiyah and Horrocks in the spirit of the theorem of Grothendieck. Finally in \secref{RH}, we give an introduction to linear homogenous differential equations or systems on a Riemann surface $M$, and explain more in depth the classical Riemann-Hilbert problem, by giving when it admits a positive solution without any restriction (regular singular situation), and by presenting a case where one encounters a counterexample to it, as given by Bolibrukh (Fuchsian case). Then finally we introduce the Birkhoff theorem (\theoref{birkh}), the proof of which relies on the Birkhoff factorization lemma (\lemref{birkho}), a consequence of which is the Grothendieck theorem, see \corref{birkhofffac}.
\section{Vector bundles on Riemann surfaces}
\label{vec}
In this preliminary section we use \cite{griffiths, forster, gunning, hitchin}. We give here the basics on Riemann surfaces and vector bundles.
\begin{definition}
A Riemann surface $M$ is a connected one dimensional complex manifold, i.e. a two dimensional real smooth manifold with a maximal set of coordinate charts $\phi_\alpha:U_\alpha\to\R^2\simeq\C$: such that the change of charts $\phi_\beta\circ\phi^{-1}_\alpha$, is an invertible holomorphic function from $\phi_\alpha(U_\alpha\cap U_\beta)$ to $\phi_\beta(U_\alpha\cap U_\beta)$, for all $\alpha,\beta$, and $M$ is connected.
\end{definition}
\begin{example}
Recall that $\mathbb P^1(\C)$ is the set lines of $\C^2$ through the origin: For $(z_0,z_1)\in\C^2\setminus\{(0,0)\}$ we define
$$\left[z_0:z_1\right]=\{\lambda(z_0,z_1),\lambda\in\C\}$$
then
$$\mathbb P^1(\C):=\{\left[z_0:z_1\right],\left(z_0,z_1\right)\in\C^2\setminus\{(0,0)\}\}.$$
Let $M=\mathbb P^1(\C)$ be the usual coordinate patches $U_0=\{\left[z_0:z_1\right]\in\mathbb P^1(\C),\,z_0\not=0\}$ and $U_1=\{\left[z_0:z_1\right]\in\mathbb P^1(\C),\,z_1\not=0\}$. Then we have the following two homeomorphisms
$$\psi_0:U_0\to\C,\;\left[z_0:z_1\right]\mapsto \frac{z_1}{z_0}$$
and

$$\psi_1:U_1\to\C,\;\left[z_0:z_1\right]\mapsto \frac{z_0}{z_1};$$
on the overlap we have
$$\psi_1\circ\psi_0^{-1}:\C^\times\to\C^\times,\;\;z\mapsto\frac{1}{z},$$
which is holomorphic.
\end{example}

\begin{definition}
\label{def2}
A rank $n\geqslant1$ vector bundle over a Riemann surface $M$ is a complex manifold $E$ of dimension $n+1$ with a holomorphic projection $\pi:E\to M$ such that
\begin{itemize}
\item for each $z\in M$, $\pi^{-1}(z)$ is an $n$-dimensional complex vector space.
\item each point $m\in M$ has a neighborhood $U$ and a homeomorphism $\psi_U$ such that
\begin{center}
\tikzstyle{grisEncadre}=[dashed]
\tikzstyle{nfEncadre}=[thick]
\tikzstyle{nffEncadre}=[very thick]
\tikzstyle{ddEncadre}=[densely dotted]
\tikz \node [scale=1, inner sep=0] {
\begin{tikzcd}
	&\pi^{-1}(U)\arrow{d}{\pi}\arrow["{\psi_U}"]{r}{}& U\times\C^n\arrow[dl,"{pr_U}" description]\\
	&U\end{tikzcd}};
	\end{center}
is commutative.
\item The transition functions $\psi_V\circ\psi_U^{-1}$ are of the form
$$(z,w)\mapsto(z,g_{VU}(z)w)$$
where $g_{VU}\colon U\cap V\to \rm{GL}(n,\C)$ is a holomorphic map to the space of invertible $n\times n$ matrices. 
\end{itemize}
A vector bundle of rank $1$, is called a line bundle.
\end{definition}

\begin{example}
We give examples of line bundles. We fix $p\in M$, $U_0$ a neighborhood of $p$ with coordinate chart $z$ such that $z(p)=0$. Let $U_1=M\setminus\{p\}$. Then we can use $z$ as a transition function to define a line bundle on $M$ since $z=g_{01}$ is holomorphic and non-vanishing on $U_0\cap U_1=U_0\setminus\{p\}$. We patch together $U_0\times\C$ and $U_1\times\C$ over $U_0\cap U_1$ by using $\psi$ defined by
$$\psi(m,w)=(m,g_{01}(m)w).$$
This gives for each point $p\in M$ a line bundle which we denote by $L_p$, called the point bundle associated to $p$. 

Let $M=\mathbb P^1(\C)$ with the usual coordinate patches $U_0=\{\left[z_0:z_1\right]\in\mathbb P^1(\C),\,z_0\not=0\}$ and $U_1=\{\left[z_0:z_1\right]\in\mathbb P^1(\C),\,z_1\not=0\}$. Then the transition function $g_{01}(z)=z^n$ on $U_0\cap U_1=\C^\times$ defines a line bundle on $\mathbb P^1(\C)$, usually denoted $\mathcal O_{\scaleto{\mathbb P^1(\C)}{7pt}}(k)$, for each $k\in\Z$. 

\end{example}
\begin{definition}
A holomorphic section of a holomorphic vector bundle $E$ over $M$ is a holomorphic map $s:M\to E$ such that $\pi\circ s=\identity_M$.
\end{definition}
In local trivializations $\psi_U$, $\psi_V$ of the vector bundle $E$, the section gives
$$\psi_U(s(m))=(m,s_U(m))\in U\times \C^n$$
and
$$\psi_V(s(m))=(m,s_V(m))\in V\times \C^n$$
hence on the overlap $U\cap V$, we have
$$\psi_V\circ\psi_U^{-1}(m,s_U(m))=(m,s_V(m))=(m,g_{VU}s_U(m)).$$
This gives
$$s_V=g_{VU}s_U,\,\text{on}\;U\cap V.$$
One can add sections pointwise
$$(s+t)(m):=s(m)+t(m)$$
and multiply sections by scalars
$$(\lambda s)(m)=\lambda s(m),$$
so the space of all sections of $E$ is a vector space, denoted $H^0(M,E)$. We have the following important result \cite{forster}, \cite[chap.~0]{griffiths}.
\begin{theorem}
If $M$ is a compact Riemann surface, $H^0(M,E)$ is finite dimensional.
\end{theorem}
\begin{definition}
If $M$ is a compact Riemann surface, its genus $g_{\scaleto{M}{4pt}}$ is defined to be $\dim_\C H^0(M,K_{\scaleto{M}{4pt}})$, where $K_{\scaleto{M}{4pt}}$ is the canonical line bundle of $M$.
\end{definition}
\begin{example}
Every vector bundle on an arbitrary Riemann surface $M$ has at least one section, namely the one whose value at every point of $M$ is 0. It is called the zero section.

Let $M=\mathbb P^1(\C)$. A section of its canonical bundle $K_{\scaleto{\mathbb P^1(\C)}{7pt}}$ looks like $f_0(z)dz$ on $U_0$ and $f_1(\tilde{z})d\tilde{z}$ on $U_1$ where $f_0$ and $f_1$ are holomorphic functions on $\C$. These forms must agree on the overlap $U_0\cap U_1=\C^\times$, $\tilde{z}=\frac{1}{z}$. This gives
$$d\tilde{z}=-z^{-2}dz,\;f_0(z)dz=-z^{-2}f_1(z^{-1})dz\Longrightarrow f_0(z)=f_1(\tilde{z})=0$$
so $ g_{\scaleto{M}{4pt}}=0.$
Since $d\tilde{z}=z^{-2}(-dz)$, we have $K_{\scaleto{\mathbb P^1(\C)}{7pt}}=\mathcal O_{\scaleto{\mathbb P^1(\C)}{7pt}}(-2)$.

More generally a section of the line bundle $\mathcal{O}_{\scaleto{\mathbb P^1(\C)}{7pt}}(k)$, $k\in\Z$, is given by holomorphic functions $s_0$ and $s_1$ on $\C$ such that
$$s_0(z)=z^ks_1(\tilde{z}),\;\tilde{z}=\frac{1}{z}$$
on the overlap. Expanding these functions in their respective local coordinates, and using $\tilde{z}=z^{-1}$ we obtain
$$\sum_{l\geqslant0}a_lz^l=z^k\sum_{l\geqslant0}\widetilde{a_l}z^{-l}.$$
Equating coefficients we get
$$\widetilde{a_l}=a_l=0,\;l>k$$
and
$$\widetilde{a_0}=a_k,\,\widetilde{a_1}=a_{k-1},\ldots,\widetilde{a_k}=a_0.$$
Thus the section is given by a polynomial of degree $\leqslant k$ if $k\geqslant0$: $$\sum_{0}^ka_mz^m$$
and zero otherwise. Hence the dimension of $H^0(\mathbb P^1(\C),\mathcal O_{\scaleto{\mathbb P^1(\C)}{7pt}}(k))=k+1$, for $k\geqslant0$, and is equal to zero when $k<0$.
\end{example}

If $\mathcal S$ is a sheaf on $M$, we can construct the Cech cohomology groups $H^p(M,\mathcal S)$ with coefficients in $\mathcal S$ as follows. For a (locally finite) covering $\mathscr U=\{U_\alpha\}_{\alpha\in A}$ of $M$ by open sets, we introduce
$$\mathcal S^0=\oplus_\alpha\mathcal S(U_\alpha)$$
\begin{equation*}
\begin{split}
\mathcal S^1=&\oplus_{\alpha\not=\beta}\mathcal S(U_\alpha\cap U_\beta)\\
&\vdots\\
\mathcal S^p=&\oplus_{\alpha_0\not=\ldots\not=\alpha_p}\mathcal S(U_{\alpha_0}\cap\ldots U_{\alpha_p})
\end{split}
\end{equation*}
and define $\mathcal C^p$ to be the alternating elements in $\mathcal S^p$. This means that for a permutation of the indices the open set does not change but one multiplies the section on that set by the signature of the permutation.

Define the homomorphism of Abelian groups $\partial:\mathcal C^p\to \mathcal C^{p+1}$ by
$$(\partial f)_{\alpha_0\ldots\alpha_{p+1}}=\sum_{i=0}^{p+1}(-1^if_{\alpha_0\ldots\hat{\alpha_i}\ldots\alpha_{p+1}}\bigg|_{U_{\alpha_0\cap\ldots\cap U_{\alpha_{p+1}}}}.$$
Here $f_{\alpha_0\ldots\hat{\alpha_i}\ldots\alpha_{p+1}}\in\mathcal S(U_{\alpha_0}\cap\ldots U_{\alpha_{i-1}}\cap U_{\alpha_{i+1}}\ldots U_{\alpha_{p+1}})
$. $\partial$ is called the boundary operator and one has $\partial^2=0$. 

\begin{definition}
The $p$-th cohomology group of $M$ with coefficients in $\mathcal S$, relative to the covering $\mathscr U=(U_\alpha)_\alpha$, is
$$H_{\mathscr U}^p(M,\mathcal S):=\frac{\kernel\partial:\mathcal C^p\to \mathcal C^{p+1}}{\image:\mathcal C^{p-1}\to\mathcal C^p}.$$
The $p$-th cohomology of $M$ with coefficients in $\mathcal S$ is then the direct limit over the open covers $\mathscr U$ of $M$, partially ordered by refinement, of the $H_{\mathscr U}^p(M,\mathcal S)$:
$$H^p(M,\mathcal S)=\varinjlim_{\text{$\mathscr U$ covering of M}}H_{\mathscr U}^p(M,\mathcal S).$$
\end{definition}




\begin{example}
For $L$ a line bundle, with transition functions $g_{\alpha\beta}=\psi_\alpha\circ\psi_{\beta}^{-1}$, then $g_{\alpha\beta}=g_{\beta\alpha}^{-1}$; so the family $g=(g_{\alpha\beta})$ lies in $\mathcal C^1$ for the sheaf $\mathcal O_{\scaleto{M}{4pt}}^\times$ of non-vanishing holomorphic functions. Furthermore
$$(\partial g)_{\alpha\beta\gamma}=g_{\beta\gamma}g_{\alpha\gamma}^{-1}g_{\alpha\beta}=\identity,$$
so $g=(g_{\alpha\beta})_{\alpha\beta}\in\kernel\partial$. Hence it defines an element of $H^1(M,\mathcal O_{\scaleto{M}{4pt}}^\times)$. Actually one can show that the set of isomorphism classes of line bundles on $M$ is $H^1(M,\mathcal O_{\scaleto{M}{4pt}}^\times)$.
\end{example}
For the following theorems we refer to \cite{gunning, forster, griffiths}.
\begin{theorem}[Vanishing theorem]
\label{vanish1}
Let $M$ be a Riemann surface. If $S=\mathcal O (E)$ is the sheaf of holomorphic sections of the vector bundle $E$, then $H^p(M,\mathcal S)=0$ for $p>1$. If $\mathcal S=\C$ or $\Z$, then $H^p(M,\mathcal S)=0$ for $p>2$. 
\end{theorem}
We set $H^1(M,\mathcal O(E))=H^1(M,E)$.
\begin{theorem}[Serre duality for vector bundles]
If $M$ is a compact Riemann surface we have
$$H^{1}(M,E)\simeq H^0(M,K_{\scaleto{M}{5pt}}\otimes E^\star)^\star$$
and
$$H^{0}(M,E)\simeq H^1(M,K_{\scaleto{M}{5pt}}\otimes E^\star)^\star,$$
where $E^\star$ is the line bundle dual to $E$.
\end{theorem}

Let us start with the exponential exact sequence on a compact Riemann surface $M$
\begin{center}
\tikzstyle{grisEncadre}=[dashed]
\tikzstyle{nfEncadre}=[thick]
\tikzstyle{nffEncadre}=[very thick]
\tikzstyle{ddEncadre}=[densely dotted]
\tikz \node [scale=1, inner sep=0] {
\begin{tikzcd}
    0\arrow{r} & \mathbb{Z}\arrow[hookrightarrow]{r}& \mathcal{O}_{\scaleto{M}{5pt}}\arrow{r}{e(2i\pi\,\cdot)} & \mathcal{O}_{\scaleto{M}{5pt}}^\times\arrow{r}& 1 
\end{tikzcd}};
\end{center}
where $e(2i\pi\,\cdot):f\mapsto \exp(2i\pi f)$, $f$ a section of $\mathcal O_{\scaleto{M}{5pt}}$.

By homological algebra this gives rise to a long exact sequence in cohomology
\begin{center}
\tikzstyle{grisEncadre}=[dashed]
\tikzstyle{nfEncadre}=[thick]
\tikzstyle{nffEncadre}=[very thick]
\tikzstyle{ddEncadre}=[densely dotted]
\tikz \node [scale=0.8, inner sep=0] {
\begin{tikzcd}
	\displaystyle0 \arrow[r] &\Z\arrow[hookrightarrow]{r}&\C\arrow{r}{}\arrow[d, phantom,""{coordinate, name=Z}&\C^\times\arrow{r}{}] &\C^\times\arrow[dll,"", rounded corners, to path={--([xshift=2ex]\tikztostart.east)|-(Z) [near end]\tikztonodes -| ([xshift=-2ex]\tikztotarget.west)--(\tikztotarget)}]\\
	{}{}&
	H^1(M,\Z)\arrow{r}{}&H^1(M,\mathcal O_{\scaleto{M}{5pt}})\arrow{r}{}&H^1(M, \mathcal O_{\scaleto{M}{5pt}}^\times)\arrow{r}{}&H^2(M,\Z)\arrow{r}{}&H^2(M,\mathcal O_{\scaleto{M}{5pt}})\arrow{r}{}&\ldots
\end{tikzcd}};
\end{center}

The first part of this sequence follows from the fact that the global holomorphic functions on a compact Riemann surface are constant. Since $\exp$ is surjective onto $\C^\times$, then by exactness we get the injection $H^1(M,\mathcal \Z)\hookrightarrow H^1(M,\mathcal O_{\scaleto{M}{5pt}})$. Also $H^2(M,\mathcal O_{\scaleto{M}{5pt}})=0$ by \theoref{vanish1}; so the previous long exact sequence reduces to the short exact sequence
\begin{center}
\tikzstyle{grisEncadre}=[dashed]
\tikzstyle{nfEncadre}=[thick]
\tikzstyle{nffEncadre}=[very thick]
\tikzstyle{ddEncadre}=[densely dotted]
\tikz \node [scale=1, inner sep=0] {
\begin{tikzcd}
    0\arrow{r} & \frac{H^1(M,\mathcal O_{\scaleto{M}{4pt}})}{H^1(M,\mathcal \Z)}\arrow[hookrightarrow]{r}& H^1(M,\mathcal O_{\scaleto{M}{4pt}}^\times)\arrow{r}{\delta} & H^2(M,\mathcal \Z)\arrow{r}& 0
\end{tikzcd}};
\end{center}
Since $M$ is a two dimensional compact oriented connected manifold, we have by Poincaré duality \cite{griffiths}$$H^2(M,\mathcal \Z)\simeq\Z.$$

\begin{definition}
The degree of the line bundle $L$ is $\delta(\left[L\right])$. It is denoted $\deg L$. The degree depends only on the isomorphism class $\left[L\right]$, the class of $L$ in $H^1(M,\mathcal O_{\scaleto{M}{5pt}}^\times)$. \end{definition}

\begin{definition}
If $E$ is a vector bundle on $M$, we define its degree by
$$\deg(E):=\deg(\det(E))$$
as the degree of the line bundle $\det(E):=\bigwedge^{{\scaleto{\rank(E)}{6pt}}}(E)$.
\end{definition}
For the next theorem, one may consult \cite{griffiths}. We have
\begin{theorem}[Riemann-Roch]
If $E$ is a vector bundle on a compact Riemann surface $M$ of genus $g_{\scaleto{M}{4pt}}$, then 
$$\dim H^0(M,E)-\dim H^1(M,E)=\deg E+\rank(E)\times(1-g_{\scaleto{M}{4pt}}).$$  
\end{theorem}

\section{The Grothendieck Theorem}
\label{birkgrothen}
In this section we present the splitting theorem of Grothendieck for vector bundles on the Riemann sphere.
\begin{theorem}[Grothendieck]
\label{grothendieck}
If $E$ is a rank $n$ holomorphic vector bundle on $\mathbb P^1(\C)$, then
$$E\cong\mathcal O_{\scaleto{\mathbb P^1(\C)}{7pt}}(a_1)\oplus\ldots\mathcal O_{\scaleto{\mathbb P^1(\C)}{7pt}}(a_n)$$
for some $a_i\in\Z$. Furthermore $\mathcal O_{\scaleto{\mathbb P^1(\C)}{7pt}}(a_1)\oplus\ldots\oplus\mathcal O_{\scaleto{\mathbb P^1(\C)}{7pt}}(a_n)\cong \mathcal O_{\scaleto{\mathbb P^1(\C)}{7pt}}(b_1)\oplus\ldots\mathcal \oplus \mathcal O_{\scaleto{\mathbb P^1(\C)}{7pt}}(b_{l})$ if and only if $n=l$ and up to reordering $a_i=b_i$.\end{theorem}
\begin{proof}
We give just an idea of the proof. We use induction on the rank of the vector bundle $E$. The result is clearly true for line bundles since $H^1(\mathbb P^1(\C),\mathcal O_{\scaleto{\mathbb P^1(\C)}{7pt}}^\times)=\Z$. Then one shows that for large $k>>0$, $E(k):=E\otimes\mathcal O_{\scaleto{\mathbb P^1(\C)}{7pt}}(k)$ splits as
$$E(k)=\mathcal O_{\scaleto{\mathbb P^1(\C)}{7pt}}\oplus \mathcal{Q}$$
where $\mathcal{Q}$ is a vector bundle of rank $n-1$. In the proof of the splitting of $E(k)$ at crucial points the Riemann-Roch theorem is used. By induction
$$\mathcal Q\cong\mathcal O_{\scaleto{\mathbb P^1(\C)}{7pt}}(c_1)\oplus\ldots\oplus\mathcal O_{\scaleto{\mathbb P^1(\C)}{7pt}}(c_{n-1}).$$
Hence
$$E\cong\mathcal O_{\scaleto{\mathbb P^1(\C)}{7pt}}(-k)\oplus\mathcal O_{\scaleto{\mathbb P^1(\C)}{7pt}}(c_1-k)\oplus\ldots\oplus\mathcal O_{\scaleto{\mathbb P^1(\C)}{7pt}}(c_{n-1}-k).$$
 The uniqueness of the $a_i$ follows from the vanishing of $H^0(\mathbb P^1(\C),\mathcal O_{\scaleto{\mathbb P^1(\C)}{7pt}}(a))$ for $a<0$.
See \cite{hitchin} for more details.
\end{proof}
\begin{remark}
More generally Grothendieck has classified holomorphic principal $G$-bundles $\mathscr P$ on $\mathbb{P}^1(\C)$, for a reductive complex Lie group $G$, and the vector bundles $\mathscr V$ associated to them via a representation $\mu\colon G\to \rm{GL}(V)$ on a finite dimensional complex vector space $V$. We refer to \cite[th.~1.1, th.~1.2]{grothendieck}, and  \cite{hwang} (prop. 3 and the paragraph which follows it), for more details.
\end{remark}
\subsection{Some further works in the spirit of Birkhoff-Grothendieck}
We have the following necessary and sufficient condition of Horrocks, for splitting of vector bundles on the complex projective space of dimension $n\geqslant1$, see \cite[p.~21]{okonek}.
\begin{theorem}[Horrocks]
A holomorphic vector bundle $E$ on $\mathbb P^n(\C)$, $n\geqslant1$, splits precisely when
$$H^p(\mathbb P^n, E(k))=0,$$
for $p=1,\ldots,n-1$ and for all $k\in\Z$, where $E(k)=E\otimes \mathcal{O}_{\mathbb P^n(\C)}(k)$, where we recall that the tautological line bundle $\mathcal{O}_{\mathbb P^n(\C)}(-1)$ on $\mathbb P^n(\C)$ is
$$\mathcal{O}_{\mathbb P^n(\C)}(-1)=\{(u,v)\in\mathbb P^n(\C)\times\C^{n+1}|v\in u\}$$
and for $k\in\Z$
\[\mathcal O_{\mathbb P^n(\C)}(k)=\begin{cases}
    \mathcal{O}_{\mathbb P^n(\C)}(-1)^{\otimes k}  & \text{ for $k\geqslant0$}, \\
    \mathcal{O}_{\mathbb P^n(\C)}(-1)^{\otimes |k|}  & \text{for $|k|\leqslant0$}.
\end{cases}\]
\end{theorem}
\begin{definition}
Let $E\xrightarrow{\pi} M$ be a holomorphic vector bundle of rank $n>0$ on a complex manifold $M$. A holomorphic subbundle $F\subseteq E$ is a collection of subspaces $\{F_m\subset E_m\}_{{\scaleto{m\in M}{4pt}}}$ of the fibers $E_m=\pi^{-1}(m)$ such that $F=\cup_{{\scaleto{m\in M}{4pt}}}F_m$ is a submanifold of $E$, given by the embedding $i\colon F\hookrightarrow E$. This means that every $m\in M$ has a neighborhood $U$ and a trivialization
$$\psi_U\colon E_U\to U\times \C^n,\;E_U:={\pi^{-1}(U)},$$
such that
$$\psi_U\big|_{F_U}\colon F_U\to U\times \C^r\subset U\times\C^n,\;F_U:={(\pi\circ i)^{-1}(U)},\;r\leqslant n$$
and $i\big|_{F_m}$ is the inclusion of $F_m$ into $E_m$.
\end{definition}
\begin{definition}
A non-zero holomorphic vector bundle of rank $n>0$ on a connected complex manifold is indecomposable if it is not the direct sum of two non-zero subbundles.
\end{definition}
The following is proved in \cite{atiyah}
\begin{theorem}[Atiyah-Krull-Remak-Schmidt]
Any holomorphic vector bundle on a connected compact complex manifold $M$ is a direct sum of indecomposable subbundles. Furthermore let $E_1,\ldots,E_k$ and $F_1,\ldots,F_l$ be indecomposable holomorphic vector bundles on $M$, such that $E_1\oplus\ldots\oplus E_k$ is isomorphic to $F_1\oplus\ldots\oplus F_l$. Then $k=l$, and up to permutation of indices, $E_1,\ldots,E_k$ are isomorphic to $F_1,\ldots,F_l$, respectively.
\end{theorem}

\begin{remark}
Let $E\xrightarrow{\pi} M$ and $E^\prime\xrightarrow{\pi^\prime} M$ be two holomorphic vector bundles on the complex manifold $M$. Then $\varphi:E\to E^\prime$ is a homomorphism if and only if
\begin{center}
\tikzstyle{grisEncadre}=[dashed]
\tikzstyle{nfEncadre}=[thick]
\tikzstyle{nffEncadre}=[very thick]
\tikzstyle{ddEncadre}=[densely dotted]
\tikz \node [scale=1, inner sep=0] {
\begin{tikzcd}
	&E\arrow{d}{\pi}\arrow["{\varphi}"]{r}{}& E^\prime\arrow[dl,"{\pi^\prime}" description]\\
	&M\end{tikzcd}};
	\end{center}
is commutative.
\end{remark}
\begin{remark}
Atiyah has given the complete classification of vector bundles on elliptic curves, by describing precisely the set of isomorphy classes of indecomposable vector bundles of rank $r\geqslant1$ and of degree $d$ on an elliptic curve. This is sufficient according to Krull-Remak-Schmidt decomposition. We refer to \cite{atiyah1} for further details.
\end{remark}
\section{Riemann-Hilbert problem and Birkhoff's theorem}
\label{RH}
We explain in this section the basics of the theory of linear homogeneous ordinary differential equations or systems on a Riemann surface $M$, then we explain the problem of Riemann-Hilbert in classical terms, finally we give a presentation of the Birkhoff theorem and the Birkhoff factorization theorem, which will imply the Grothendieck theorem \theoref{grothendieck}, following \cite{sibuya}. The two theorems are in fact equivalent. Our inspiration in this section are \cite{singer, sibuya, andre, deligne, katz, forster}. For a nice survey on the origins of problems of Riemann-Hilbert type, we refer to \cite{bothner, rohrl1, zverovich}.
\begin{definition}
A linear differential equation, with meromorphic coefficients on a Riemann surface $M$, is an expression
\begin{equation}
\label{markus1}
w^{(n)}+a_{n-1}(z)w^{(n-1)}+\cdots+a_0(z)w=0, \qquad w^{(k)}:=\dfrac{d^kw}{dz^k},\;0\leqslant k\leqslant n
\end{equation}
such that
\begin{enumerate}
\item for each local coordinate system $(z)$ on $M$ there is prescribed such a linear differential equation with meromorphic coefficients in $(z)$.
\item On the overlap of two coordinate systems $(z)$ and $(\zeta)$, the prescribed $n$-th order linear differential equations have the same holomorphic solutions on each open subset of $(z)\cap(\zeta)$. 
\end{enumerate}
\end{definition}
\begin{definition}
A  linear differential equation \eqref{markus1}:
$$w^{(n)}+a_{n-1}(z)w^{(n-1)}+\cdots+a_0(z)w=0,$$
with $a_k(z)$ having poles at $z=0$ (centered chart), is called of the first kind if $b_{n-k}(z):=z^ka_{n-k}(z)$, $k=1,2,\ldots,n$ are all holomorphic at $z=0$. The highest order pole among the $b_{n-k}(z)=z^ka_{n-k}(z)$ is called the rank of \eqref{markus1}. When every singularity of \eqref{markus1} is of the first kind, one says that \eqref{markus1} is of the first kind.
\end{definition}
\begin{remark}
In local coordinates $\zeta$, with $z=z(\zeta)$, the coefficient of $\dfrac{d^{n-k}w}{d\zeta^{n-k}}$ is a polynomial of the derivatives of $\zeta(z)$, the coefficients $a_{n-1}(z(\zeta)),\ldots,a_{n-k+1}(z(\zeta))$, $a_{n-k}(z(\zeta))\left(\dfrac{d\zeta}{dz}\right)^{n-k}$, all divided by $\left(\dfrac{d\zeta}{dz}\right)^n$. Thus the property of \eqref{markus1} having a singularity of the first kind at a point does not depend on the local coordinates. Also the rank of \eqref{markus1} at a point $P$ is intrinsic.
\end{remark}
\begin{definition}
A linear homogeneous system, with meromorphic coefficients on a Riemann surface $M$, is given by an expression
\begin{equation}
\label{markus2}
\dfrac{d w^i}{dz}=A_j^i(z)w^j,\qquad i,j=1,2,\ldots,n
\end{equation}
such that 
\begin{enumerate}
\item For each local coordinate system $(z)$ a meromorphic matrix $A_j^i(z)$ is prescribed.
\item On the intersection of two coordinate systems $(z)\cap(\zeta)$, the prescribed homogeneous differential systems have the same holomorphic solution functions.
\end{enumerate}
\end{definition}
\begin{remark}
The poles, and their orders, for the coefficients $A_j^i(z)$ are intrinsic. Furthermore
$$\dfrac{dw^i}{d\zeta}=A^i_j(z(\zeta))\dfrac{dz(\zeta)}{d\zeta}w^j.$$
\end{remark}
\begin{definition}
If $A(z)$ has a pole of first order at $z=0$ then one says that the system \eqref{markus2} has a singularity of the first kind there. Otherwise \eqref{markus2} has a singularity of the second kind. The highest order pole in $zA(z)$ is the rank of \eqref{markus2}. Hence at a point $P$, the system \eqref{markus2} is analytic, or has a singularity of rank $\mu=0,1,2,\ldots$. If $\mu=0$, then $P$ is a first kind singularity. The rank does not depend on the coordinate system around $P$. 
\end{definition}
For the following definition one may consult \cite[chap.~5]{singer}.
\begin{pdefinition}
Let \eqref{markus1} or \eqref{markus2} be a linear homogenous differential equation (or first order system), with meromorphic coefficients on a Riemann surface $M$. Let $\mathscr S$ be the (isolated) singularities of \eqref{markus1} or \eqref{markus2} in $M$. Then at each ordinary, or non-singular point $P$, the solution family is holomorphic in a neighborhood of $P$ and forms a complex vector space of dimension $n$. Let $W(z)$ be the associated fundamental matrix of solutions. Consider the set $\Gamma$ of all closed loops in $M\setminus \mathscr S$, based at $P$. Each such loop $\gamma$ defines a linear transformation of the solution space onto itself, by analytic continuation. This means that there exists a matrix $\mathfrak M\in\rm{GL}(n,\C)$ such that
$$\gamma_\ast W(z)=W(z)\mathfrak M,$$
where $\ast$ denotes the operation of analytic continuation of the fundamental matrix $W(z)$ along the loop $\gamma$. The matrix $\mathfrak M$ depends only on the homotopy equivalence class $\left[\gamma\right]$ of $\gamma$ in $M\setminus \mathscr S$. We set $\mathfrak M=\mathfrak M_{\left[\gamma\right]}$. This gives a map
\begin{equation}
\label{markus3}
\begin{split}
\Psi\colon\pi_1(M\setminus \mathscr S,P)&\to \rm{GL}(n,\C)\\
&\left[\gamma\right]\mapsto \mathfrak M_{\left[\gamma\right]}.
\end{split}
\end{equation}
Now recall that the product of two loops $\gamma_1$ and $\gamma_2$, based at $P$, is the loop gotten by traversing first $\gamma_1$ then $\gamma_2$, in this order. In terms of homotopy classes we thus have $\left[\gamma_1\gamma_2\right]=\left[\gamma_1\right]\left[\gamma_2\right]$. Now we have
\begin{equation*}
\begin{split}
(\gamma_1\gamma_2)_\ast W(z)&=(\gamma_2)_\ast((\gamma_1)_\ast W(x))\\
&=(\gamma_2)_\ast W(x)\mathfrak M_{\left[\gamma_1\right]}\\
&=W(x)\mathfrak M_{\left[\gamma_2\right]}\mathfrak M_{\left[\gamma_1\right]},
\end{split}
\end{equation*}
which implies
$$\mathfrak M_{\left[\gamma_1\gamma_2\right]}=\mathfrak M_{\left[\gamma_2\right]}\mathfrak M_{\left[\gamma_1\right]}.$$
Thus
\begin{equation}
\label{markus4}
\Psi(\left[\gamma_1\gamma_2\right])=\Psi(\left[\gamma_2\right])\Psi(\left[\gamma_1\right]).
\end{equation}
Therefore $\Psi$ is an antihomomorphism of groups, called the monodromy representation of the homogeneous linear differential equation \eqref{markus1}, or system \eqref{markus2}. The image of $\Psi$ which is a subgroup of $\rm{GL}(n,\C)$, is called the monodromy group of \eqref{markus1} or \eqref{markus2}. 
\end{pdefinition}
\begin{remark}
The monodromy group of \eqref{markus1} or \eqref{markus2}, based at $P$ is represented by a subgroup of matrices of $\rm{GL}(n,\C)$ once a basis has been chosen for the solution near $P$. A change of solution basis changes the monodromy group $\mu\subset \rm{GL}(n,\C)$ of \eqref{markus1} or \eqref{markus2} to a conjugate subgroup $C\mu C^{-1}$, for a fixed $C\in\rm{GL}(n,\C)$. 

For a change of the base point $P$ to $P^\prime$ inside $M\setminus \mathscr S$, the associated monodromy group are isomorphic (but not in a canonical way). For a fixed choice of basis of solutions at $P$ and $P^\prime$, and a fixed isomorphism between the abstract monodromy groups given by a curve in $M\setminus \mathscr S$ joining $P$ and $P^\prime$, the monodromy representations are conjugate. 
\end{remark}
\subsection{Fuchsian Differential equations on Riemann surfaces}
\begin{definition}
If the system \eqref{markus2}:
$$\dfrac{d w^i}{dz}=A_j^i(z)w^j,\qquad i,j=1,2,\ldots,n$$
where $A(z):=(A_j^i(z))$ has an isolated pole at $z=0$, has a solution matrix $W(z)=S(z)z^Q$, where $S(z)$ is single-valued and meromorphic at $z=0$ and $Q$ is a constant matrix, then one says that \eqref{markus2} has a regular singularity or moderate growth at $z=0$. Otherwise \eqref{markus2} has an irregular singularity at $z=0$.
\end{definition}
\begin{definition}
A homogeneous linear differential system \eqref{markus2} with meromorphic coefficients, on a Riemann surface $M$ is called regular singular in case every one of its singularities is a regular one, and it is called Fuchsian in case each one of its singular points is a simple pole (singularity of the first kind). 
\end{definition}
\begin{theorem}[Sauvage, \cite{sauvage}]
Let \eqref{markus2} be a homogeneous linear differential system with a singularity of the first kind at $z=0$, that is $A(z)$ has a simple pole at $z=0$. Then \eqref{markus2} has a regular singularity at $z=0$. More generally if \eqref{markus2} has only singularities of the first kind on $M$, then each of its singularities is a regular singular one.
\end{theorem}
\begin{proof}
See \cite[th.~2.1, p.~111]{coddington}, \cite{sauvage}.
\end{proof}
We have \cite[p.~124]{coddington}
\begin{definition}
The differential equation \eqref{markus1}
$$w^{(n)}+a_{n-1}(z)w^{(n-1)}+\cdots+a_0(z)w=0$$
with meromorphic coefficients at $z=0$ has a regular singularity there, when every solution near $z=0$ can be expressed as a finite (constant) linear combination of terms of the form $z^r(\log z)^kp(z)$, where $r$ is a complex number, $k$ is an integer $0\leqslant k\leqslant n-1$ and $p(z)$ is analytic at $z=0$ with $p(0)\not=0$. Otherwise \eqref{markus1} has an irregular singularity at $z=0$.
\end{definition}
\begin{definition}
A linear homogeneous differential equation \eqref{markus1}
$$w^{(n)}+a_{n-1}(z)w^{(n-1)}+\cdots+a_0(z)w=0$$
has a singularity of the first kind at $z=0$ when each $b_{n-k}(z):=z^ka_{n-k}(z)$ is analytic at $z=0$. More generally if \eqref{markus1} has only singularities of the first kind on $M$, then it is called Fuchsian. \end{definition}
For the next two theorems one may consult \cite[th.~5.1, th.~5.2]{coddington}
\begin{theorem}
If a linear homogeneous differential equation \eqref{markus1}
$$w^{(n)}+a_{n-1}(z)w^{(n-1)}+\cdots+a_0(z)w=0$$
has a singularity of the first kind at $z=0$ (that is $b_{n-k}(z):=z^ka_{n-k}(z)$ is analytic at $z=0$), then \eqref{markus1} has a regular singularity at $z=0$.
\end{theorem}
\begin{theorem}[Fuchs, \cite{katz, fuchs1, fuchs2}]
A linear homogeneous differential equation 
$$w^{(n)}+a_{n-1}(z)w^{(n-1)}+\cdots+a_0(z)w=0$$
with meromorphic coefficients on a Riemann surface M, has only singularities of the first kind if and only if it has only regular singularities. 
\end{theorem}
\begin{definition}
Consider the differential system \eqref{markus2}, with singularity of the first kind at $z=0$, given in matrix form by
$$\dfrac{d w}{dz}=\left(\dfrac{R}{z}+\sum_{m\geqslant0}R_mz^mw\right),$$
with constant matrices $R$, $R_m$. The eigenvalues of $R$ are called the exponents of \eqref{markus2} at $z=0$ and the sum of these exponents is the trace of $R$, $\rm{Tr}(R)$. For an equation \eqref{markus1} with singularity of the the first kind at $z=0$, the exponents at $z=0$ are the roots of the indicial equation
$$\rho(\rho-1)\ldots(\rho-n+1)+b_{n-1}(0)\rho(\rho-1)\ldots(\rho-n+2)+\ldots+b_0(0)$$
and the sum of the exponents is
$$-b_{n-1}(0)+\dfrac{n(n-1)}{2}.$$
\end{definition}
\begin{remark}
The matrix $R$, and thus its eigenvalues are independent of the choice of local coordinates. One also shows that the $b_{n-k}(0)$ are invariant.
Similarly each of the the exponents $\rho_1$, $\rho_2$, $\ldots$, $\rho_n$ (possibly multiple) is invariant.
\end{remark}
\begin{theorem}
\label{levinson}
Consider the system
$$w^\prime=(z^{-1}R+\sum_{k\geqslant0}R_kz^k)w$$
with $R\not=0$ and $R_k$ constant $n\times n$ matrices, which is assumed to have a singularity of the first kind at $z=0$. If $R$ has eigenvalues which do not differ by positive integers, then there exists a fundamental matrix $W$ of the form
$$W(z)=S(z)z^R\quad (0<|z|<c,\;\;c>0)$$ 
where $S(z)$ is the convergent power series $S(z)=\displaystyle\sum_{k=0}^\infty S_kz^k$, with $S_0=\identity$. Any other fundamental matrix is of the form $S(z)z^R\mathscr K$, with $\mathscr K\in\rm{GL}(n,\C)$.
\end{theorem}
\begin{proof}
One may consult \cite[th.~4.1, p.~119]{coddington}.
\end{proof}
\begin{theorem}[Fuchs relations]
Let \eqref{markus2}: $\dfrac{dw}{dz}=A(z)w$ be a differential system (in matrix form) on a compact Riemann surface $M$. Assume the coefficients are meromorphic and each singularity is of the first kind (Fuchsian system). Let $E_P=\lambda_{1,P}+\ldots\lambda_{n,P}$ be the sum of the exponents for each singular point $P\in M$. Then $$\sum_{P\in M}E_P=0.$$
For a Fuchsian linear homogeneous differential equation $w^{(n)}+a_{n-1}(z)w^{(n-1)}+\cdots+a_0(z)w=0$ on a compact Riemann surface $M$ we similarly have
$$\sum_{P\in M}E_P=\dfrac{n(n-1)}{2}(N+2g_{\scaleto{M}{4pt}}-2),\qquad E_P:=\sum_{i=1}^n\rho_{\scaleto{i,P}{5pt}}$$
where $n$ is the order of \eqref{markus1}, $N$ is the number of its singularities, and $g_{\scaleto{M}{4pt}}$ is the genus of $M$.
\end{theorem}
\begin{proof}
For a differential system \eqref{markus2}, we consider $W(z)$ a fundamental solution matrix of it and denote by $D(z)$ its determinant. Then
$$\dfrac{d}{dz}(\log(D(z)))dz=\dfrac{D^\prime(z)}{D(z)}dz=\rm{Tr}(A(z))dz$$
gives a meromorphic differential (single-valued) on $M$. Moreover $\dfrac{D^\prime(z)}{D(z)}dz$ is holomorphic except at the singular points $P$ of \eqref{markus2} where it has residue $E_P$. Since the sum residues of a meromorphic differential on a compact Riemann surface is zero, we obtain
$$\sum_{P\in M}E_P=0.$$
The analysis for the case of \eqref{markus1} is similar. For instance in the case of the Riemann sphere, we need to look at the residues of the differential $a_{n-1}(z)(dz)$ at finite points, and at $\infty$. They can be expressed in terms of the sum of local exponents at the considered point, then we again use the fact that the sum of residues of a meromorphic differential on $\C\mathbb P^1$ is zero, to conclude.
\end{proof}
\begin{theorem}
Let \eqref{markus1}: $w^{(n)}+a_{n-1}(z)w^{(n-1)}+\cdots+a_0(z)w=0$ be a linear homogeneous differential equation on the complex plane. Then it defines a Fuchsian differential equation on the Riemann sphere $\mathbb P^1(\C)$ if and only if each coefficient $a_{n-k}(z)$ is a rational function (with poles of order $\leqslant k$) and furthermore
$$|z^ka_{n-k}(z)|=|b_{n-k}(z)|<B,\qquad z\to\infty$$
for some bound $B$, i.e.
$$a_{n-k}(z)=O(1/z^k),\qquad z\to\infty.$$ 
\end{theorem}
\begin{theorem}
Let
$$\dfrac{d w^i}{dz}=A_{1j}^i(z)w^j,\quad \dfrac{d w^i}{dz}=A_{2j}^i(z)w^j, \quad i,j=1,2,\ldots,n$$
be meromorphic differential systems on a compact Riemann surface $M$, with the same singular points $\mathscr S$ which are all of the first kind. Assume that at each singularity the two systems have the same exponents, no two of which differ by an integer. Further assume that the two systems have the same monodromy group in $\rm{GL}(n,\C)$, relative to a basis of solution which reduces to the identity at the base point $P\in M\setminus \mathscr S$. Then the two differential systems are the same. 
\end{theorem}
\begin{proof}
Let $\mathscr S$ be the singular points of the two systems. Let $W_1(z)$ and $W_2(z)$ be the solutions matrices of the first system, respectively the second system, which reduce to $\identity$ at the base point $P\in M\backslash\mathscr S$. We consider the matrix $W_1(z)W_2^{-1}(z)$, and continue it analytically along all curves in $M\backslash \mathscr S$. We find after analytic continuation around a loop
\begin{equation*}
\begin{split}
W_1(z)\to W_1(z)C_1\\
W_2(z)\to W_2(z)C_2.
\end{split}
\end{equation*}
But the monodromy is the same for the two systems around this loop and hence we have $C_1=C_2$.

Thus 
$$W_1(z)C_1C_1^{-1}W_2(z)^{-1}=W_1(z)W_2(z)^{-1}$$
and hence $W_1(z)W_2(z)^{-1}$ is single-valued and holomorphic on $M\backslash\mathscr S$. Next we examine $W_1(z)W_2(z)^{-1}$ in a neighborhood of a singular point $P$. Here (\theoref{levinson})
\begin{equation*}
\begin{split}
W_1(z)=(I+zS_{1,1}+z^2S_{2,1}+\ldots)z^{R_1}\mathscr{K}_1\\
W_2(z)=(I+zS_{1,2}+z^2S_{2,2}+\ldots)z^{R_2}\mathscr{K}_2
\end{split}
\end{equation*}
where $R_1$ and $R_2$ (which are diagonalizable by hypothesis) are each similar to the matrix $\Lambda:=\diag(\lambda_1,\lambda_2,\ldots,\lambda_n)$, say
\begin{equation*}
\begin{split}
Q_1R_1Q_1^{-1}=\Lambda\\
Q_2R_2Q_2^{-1}=\Lambda.
\end{split}
\end{equation*}
So we have
\begin{equation*}
W_1(z)=(I+zS_{1,1}+z^2S_{2,1}+\ldots)Q_1^{-1}z^\Lambda\widehat{\mathscr K_1},\quad\widehat{\mathscr K_1}=Q_1\mathscr K_1
\end{equation*}
where $z^\Lambda=\diag(z^{\lambda_1},z^{\lambda_2},\ldots,z^{\lambda_n})$. Similarly we find
\begin{equation*}
W_2(z)=(I+zS_{1,2}+z^2S_{2,2}+\ldots)Q_2^{-1}z^\Lambda\widehat{\mathscr K_2},\quad\widehat{\mathscr K_2}=Q_1\mathscr K_2,
\end{equation*}
with $z^\Lambda$ as before. Thus
$$W_1(z)W_2(z)^{-1}=(I+zS_{1,1}+z^2S_{2,1}+\ldots)Q_1^{-1}z^\Lambda\widehat{\mathscr K_1}\widehat{\mathscr K_2}^{-1}z^{-\Lambda}Q_2(I-zS_{1,1}+z^2S_{2,1}+\ldots).$$
Now we show that $\widehat{\mathscr K_1}\widehat{\mathscr K_2}^{-1}$ commutes with $z^\Lambda$. Then $W_1(z)W_2(z)^{-1}$ is holomorphic at $P$, hence holomorphic everywhere, hence constant since $M$ is a compact Riemann surface. This gives 
\begin{equation*}
\begin{split}
W_1(z)W_2^{-1}(z)=\identity\Longleftrightarrow W_1(z)=W_2(z) 
\end{split}
\end{equation*}
on $M$. Then by interpreting the two differential systems as regular singular connections on the trivial vector bundle $\mathcal{O}_{\scaleto{M}{5pt}}^n$, we see that they give rise to the same local system \cite{deligne}; using the Riemann-Hilbert correspondence \cite[cor.~3.3, p.~103]{sabbah}, we see that they are equal.

Let us then explain what happens in a neighborhood of a singular point $P$. After making a small loop around $P$, we obtain
\begin{equation*}
\begin{split}
W_1(z)\to(I+zS_{1,1}+z^2S_{2,1}+\ldots)Q_1^{-1}z^\Lambda e^{2i\pi\Lambda}\widehat{\mathscr K_1}\\
W_2(z)\to(I+zS_{1,2}+z^2S_{2,2}+\ldots)Q_2^{-1}z^\Lambda e^{2i\pi\Lambda}\widehat{\mathscr K_2}.
\end{split}
\end{equation*}
Therefore around $P$ we obtain the monodromy matrix
$$\widehat{\mathscr K_1}^{-1}e^{2i\pi\Lambda}\widehat{\mathscr K_1}$$
for the first system, respectively
$$\widehat{\mathscr K_2}^{-1}e^{2i\pi\Lambda}\widehat{\mathscr K_2}$$
for the second system. But these matrices are the same so
$$\widehat{\mathscr K_1}^{-1}e^{2i\pi\Lambda}\widehat{\mathscr K_1}
=\widehat{\mathscr K_2}^{-1}e^{2i\pi\Lambda}\widehat{\mathscr K_2}.$$
Thus
$$(\widehat{\mathscr K_1}\widehat{\mathscr K_2}^{-1})^{-1}e^{2i\pi\Lambda}(\widehat{\mathscr K_1}\widehat{\mathscr K_2}^{-1})=e^{2i\pi\Lambda}.$$
Now $e^{2i\pi\Lambda}$ is diagonal, with distinct eigenvalues, thus $\widehat{\mathscr K_1}\widehat{\mathscr K_2}^{-1}$ is also diagonal. Therefore it commutes with $z^\Lambda$.
\end{proof}
\begin{definition}Let $M$ be a Riemann surface and $\mathscr S$ a closed discrete (possibly empty) set of points in $M$. Let $P\in M\backslash \mathscr S$ and 
$$\Psi:\pi_1(M\backslash \mathscr S)\to \rm{GL}(n,\C)$$
a homomorphism. The generalized Riemann-Hilbert problem consists in finding a differential system $\dfrac{dw}{dz}=A(z)w$, with $A(z)$ a $n\times n$ holomorphic matrix on $M\backslash \mathscr S$, (or if $\mathscr S$ is empty and $M$ compact $A(z)$ is meromorphic) having the prescribed singularities (possibly essential) at points of $\mathscr S$ and the prescribed monodromy. This means that the fundamental solution matrix which reduces to the identity at $P$ yields the given representation of the fundamental group.\end{definition}
\begin{theorem}
Let $M$ be a Riemann surface, $S$ a closed (possibly empty) discrete subset of $M$ such that $M\backslash \mathscr S$ is non-compact. Further let $\rm{GL}(n,\mathbb C)$ be the linear complex Lie group and 
$$\Psi:\pi_1(M\backslash \mathscr S)\to \rm{GL}(n,\mathbb C)$$
a representation of the fundamental group of $M\backslash \mathscr S$ into $\rm{GL}(n,\mathbb C)$. Then there exists a differential system holomorphic on $M\backslash \mathscr S$ and with singularities only on $\mathscr S$, admitting the given monodromy representation $\Psi:\pi_1(M\backslash \mathscr S)\to \rm{GL}(n,\C)$.
\end{theorem}
\begin{proof}
We only give the main steps. They consists in 
\begin{itemize}
\item Construction of a principal $\rm{GL}(n,\C)$-bundle $\mathscr P$ over $M\backslash \mathscr S$.
\item Proof of the existence of a non-trivial holomorphic section of $M\backslash \mathscr S$ into $\mathscr P$. 
\item Once the existence of such a section from $M\backslash \mathscr S$ into $\mathscr P$ is proven, use it to construct a fundamental solution matrix (which is a holomorphic function from the universal cover $\widetilde{M\backslash \mathscr S}$ of $M\backslash \mathscr S$ to $\rm{GL}(n,\C)$.
\end{itemize}
See \cite{rohrl, sabbah}.
\end{proof}
\begin{theorem}
Let $M$ be a non-compact Riemann surface and $\mathscr S$ be a closed discrete (possibly empty) set of points of $M$. Let $\Psi:\pi_1(M\backslash \mathscr S)\to \rm{GL}(n,\C)$ be a prescribed homomorphism for the base point $P\in M\backslash \mathscr S$. Then there exists a regular singular differential system $\dfrac{dw}{dz}=A(z)w$ on $M$, with the prescribed singularities and the prescribed monodromy group.
\end{theorem}
\begin{proof}
\cite[chap.~3]{forster}.
\end{proof}
\begin{theorem}
Let $M$ be a compact Riemann surface and $\mathscr S$ a discrete closed (possibly empty) set of points of $M$. Let
$$\Psi:\pi_1(M\backslash \mathscr S)\to \rm{GL}(n,\C)$$
be a homomorphism, for the base point $P\in M\backslash \mathscr S$. Then there exists a regular singular differential system $\dfrac{dw}{dz}=A(z)w$ on $M$, with singularity only at $\mathscr S$ (and at one additional point if $\mathscr S$ is empty), and with the prescribed monodromy.
\end{theorem}
\begin{proof}
\cite{andre, deligne, katz}.
\end{proof}
\subsection{The Bolibrukh counterexamples: an instance}
In this paragraph we give following \cite{boli1, beauville} a counterexample to Riemann-Hilbert problem, in the Fuchsian case. 

After the works of Plemelj and Birkhoff \cite{plemelj1, plemelj, birkhoff3}, it was widely believed that the Riemann-Hilbert problem for Fuchsian (simple poles) differential systems on the Riemann sphere $\mathbb P^1(\C)$ was solved. This was the case until Bolibrukh came in \cite{boli1} with a counterexample (in fact he gave many of them over the years). In order to explain what the issue was with the proofs of Plemelj and Birkhoff, let us summarize the main steps of their proofs. Let $\mathscr S$ be a finite non empty subset of $\mathbb P^1(\C)$ and
$$\Psi:\pi_1(\mathbb P^1(\C)\setminus\mathscr S)\to\rm{GL}(n,\C),\,n\geqslant1,$$
a fixed representation. Plemelj and Birkhoff proceeded as follows; firstly following Riemann they construct a multiform invertible matrix $W(z)$ which transforms like $W(z)\Psi(\gamma)$ under analytic continuation along a loop $\gamma$. By construction the matrix $W^\prime(z)W^{-1}(z)$ is then invariant under the monodromy operation, hence it arises from a matrix $A(z)$ which is holomorphic in $\mathbb P^1(\C)\setminus\mathscr S$. Thus $W(z)$ is a fundamental matrix of the system $w^\prime(z)=A(z)w(z)$ on $\mathbb P^1(\C)\setminus\mathscr S$. The monodromy representation of this system is by construction given by $\Psi$. Furthermore, the constructed matrix $W(z)$ has moderate growth in the neighborhood of the points of $\mathscr S$. Therefore the same holds for $A(z)$, which as a consequence is meromorphic on $\mathbb P^1(\C)$. Thus we have realized $\Psi$ as the monodromy representation of a system with regular singularities. After that step, Plemelj and Birkhoff as a second step fix a point $P\in\mathscr S$, and show that $W(z)$ may be so chosen to have simple poles at all the points of $\mathscr S\setminus \{P\}$. Finally as a last step they modify the matrix $W(z)$ in order to make the resulting system also have a simple at $P$. It is during this last step that the argument breaks. Indeed Plemelj and Birkhoff assume (implicitly) that the monodromy matrix $\mathfrak M_{\left[\gamma_{\scaleto{P}{3.5pt}}\right]}$, which is the image under $\Psi$ of the homotopy class of a small loop $\gamma_{\scaleto{P}{4pt}}$ encircling once counter clockwise $P$ alone, is diagonalizable. We have the following
\begin{theorem}[\cite{beauville, boli4}]
\label{bolibeauville}
Let $\Psi:\pi_1(\mathbb P^1(\C)\setminus\mathscr S)\to\rm{GL}(n,\C)$, $n\geqslant4$ be a representation which satisfies the following conditions
\begin{itemize}
\item $\Psi$ is not irreducible, meaning that the image of $\Psi$ fixes a subspace of $\C^n$ distinct from $\{0\}$ and $\C^n$ itself;
\item each of the matrices $\mathfrak M_{\left[\gamma_{\scaleto{P}{3.5pt}}\right]}$, $P\in\mathscr S$, possesses only one eigenvalue $\mu_{\scaleto{P}{4pt}}$ and admits only one Jordan block, and furthermore for $\rm{I}_n$ the identity $n\times n$ complex matrix
$$\displaystyle\prod_{P\in\mathscr S}\mathfrak M_{\left[\gamma_{\scaleto{P}{3.5pt}}\right]}=\mathrm{I}_n;$$
\item and$$\displaystyle \prod_{P\in\mathscr S}\mu_{\scaleto{P}{4pt}}\not=1.$$
\end{itemize}
Then $\Psi$ is not isomorphic to the monodromy representation of a differential system on $\mathbb P^1(\C)$ with only simple poles.
\end{theorem}
\begin{example}[\cite{beauville, boli4}]We set $n=4$ and $\mathscr S=\{P_1,P_2,P_3\}$ where the $P_i$ are distinct, and define the representation $\Psi:\pi_1(\mathbb P^1(\C)\setminus\mathscr S)\to\rm{GL}(4,\C)$ by
\begin{equation*}
\begin{split}
&\mathfrak M_{\left[\gamma_{\scaleto{P_1}{4.2pt}}\right]}=\left(\begin{array}{cccc}1 & 1 & 0 & 0 \\0 & 1 & 1 & 0 \\0 & 0 & 1 & 1 \\0 & 0 & 0 & 1\end{array}\right),\quad \mathfrak M_{\left[\gamma_{\scaleto{P_2}{4.2pt}}\right]}=\left(\begin{array}{cccc}3 & 1 & 1 & -1 \\-4 & -1 & 1 & 2 \\0 & 0 & 3 & 1 \\0 & 0 & -4 & -1\end{array}\right),\\  &\mathfrak M_{\left[\gamma_{\scaleto{P_3}{4.2pt}}\right]}=\left(\begin{array}{cccc}-1 & 0 & 2 & -1 \\4 & -1 & 0 & 1 \\0 & 0 & -1 & 0 \\0 & 0 & 4 & -1\end{array}\right).
\end{split}
\end{equation*}
It is clear that this representation is not irreducible since its image fixes for instance the linear span of $(1,0,0,0)^t$ and $(0,1,0,0)^t$. $\mathfrak M_{\left[\gamma_{\scaleto{P_1}{4.2pt}}\right]}$ is its own Jordan block decomposition, whereas the Jordan block decomposition of $\mathfrak M_{\left[\gamma_{\scaleto{P_2}{4.2pt}}\right]}$, respectively $\mathfrak M_{\left[\gamma_{\scaleto{P_3}{4.2pt}}\right]}$ is given by
$\left(\begin{array}{cccc}1 & 1 & 0 & 0 \\0 & 1 & 1 & 0 \\0 & 0 & 1 & 1 \\0 & 0 & 0 & 1\end{array}\right)$, resp. $\left(\begin{array}{cccc}-1& 1 & 0 & 0 \\0 & -1 & 1 & 0 \\0 & 0 & -1 & 1 \\0 & 0 & 0 & -1\end{array}\right)$. 

Besides one easily verifies that one has $\mathfrak M_{\left[\gamma_{\scaleto{P_1}{4.2pt}}\right]}\mathfrak M_{\left[\gamma_{\scaleto{P_2}{4.2pt}}\right]}\mathfrak M_{\left[\gamma_{\scaleto{P_3}{4.2pt}}\right]}=\rm{I}_4$.
Thus $\Psi$ satisfies all the requirements of \theoref{bolibeauville}, and we can conclude that it is not isomorphic to the monodromy representation of a Fuchsian differential system on $\mathbb P^1(\C)$.
\end{example}
\subsection{Birkhoff's theorem}
Let 
\begin{equation*}\begin{split}\mathscr D=\{z\in\C;\;|z|>R_0,\;R_0>0\}\\\Delta\left(0,\dfrac{1}{R_0}\right)=\{z\in\C;\;|z|<R_0\}.\end{split}
\end{equation*}

We also denote by $\mathcal O\left(\Delta\left(0,\dfrac{1}{R_0}\right)\right)$, the ring of holomorphic functions on the open disc $\Delta\left(0,\dfrac{1}{R_0}\right)$. \\

We have, see \cite[chap.~3]{sibuya}, the following
\begin{theorem}[G. D. Birkhoff]
\label{birkh}
Given an integer $k\in\Z$ and a matrix $\mathcal A(\zeta)$ belonging to $\rm{M}\left(n,\mathcal O\left(\Delta\left(0,\dfrac{1}{R_0}\right)\right)\right)$, we consider the differential system
\begin{equation}
\label{sibuya1}
\dfrac{dw}{dz}=z^k\mathcal A\left(\dfrac{1}{z}\right)w,\qquad z\in \mathscr D,
\end{equation}
where $w$ is a column vector. Then, there exists a matrix $\mathcal P(\zeta)\in\rm{GL}\left(n,\mathcal O\left(\Delta\left(0,\dfrac{1}{R_0}\right)\right)\right)$ such that the transformation
$$w=\mathcal P\left(\dfrac{1}{z}\right)v$$
changes \eqref{sibuya1} into the differential system
\begin{equation}
\label{sibuya2}
\dfrac{dv}{dz}=z^k\mathcal B\left(\dfrac{1}{z}\right)v
\end{equation}
with a matrix $\mathcal B\left(\dfrac{1}{z}\right)$ whose entries are polynomials in $\dfrac{1}{z}$ with complex constant coefficients. Furthermore one can choose such a $\mathcal P(\zeta)$ so that $z=0$ is at worst a singularity of the first kind of \eqref{sibuya2}.
\end{theorem}
The main step in the proof of the \theoref{birkh} is the Birkhoff factorization lemma, a consequence of which (as we will see) is the Grothendieck theorem \theoref{grothendieck}. The following holds, see \cite[Lem. 3.31]{birkhoff2}, \cite[sec.~3]{sibuya}
\begin{lemma}[Birkhoff's factorization]
\label{birkho}
Assume that a matrix\;\,$\mathcal T(z)$ belongs to $\rm{GL}(n,\mathcal{O}(\mathscr D))$ where
$$\mathscr D=\{z\in\C;\;|z|>R_0,\;R_0>0\}.$$
Then $\mathcal T(z)$ can be written in the following form
\begin{equation}
\label{birkhof}
\mathcal T(z)=\mathcal P\left(\dfrac{1}{z}\right)\Lambda(z)\mathcal E(z)
\end{equation}
where
\begin{enumerate}
\item $\mathcal P(\zeta)\in \rm{GL}\left(n,\mathcal O\left(\Delta\left(0,\dfrac{1}{R_0}\right)\right)\right)$, $R_0>0$.
\item $\mathcal E(z)\in\rm{GL}(n,\mathcal{O}(\C))$.
\item $\Lambda(z)$ is a diagonal matrix
\[
  \Lambda(z) =
  \begin{bmatrix}
    z^{k_1} & & &\\
      &z^{k_2} & &\\
    &  &\ddots& \\
    & & &z^{k_n}
  \end{bmatrix},
\]
\end{enumerate}
and $k_1,k_2,\ldots, k_n$ are integers which one can choose such that $k_1\geqslant k_2\geqslant\ldots k_n$.
\end{lemma}
Now we recall the following well-known lemma \cite{forster, gunning}
\begin{lemma}
\label{steenrod}
Let $\mathscr V$ and $\mathscr V^\prime$ be two holomorphic vector bundles over a Riemann surface $M$, of the same rank $n>0$, and with the same coordinate neighborhoods (local trivializations) $\{V_j\}_{j\in J}$. Let $g_{ji}$, $g_{ji}^\prime$ denote their transition functions (see third bullet point in \defiref{def2}). Then $\mathscr V$ and $\mathscr V^\prime$ are isomorphic if and only if there exists holomorphic functions $\lambda_j:V_j:\rm{GL}(n,\C)$, defined for each $j\in J$, and such that
$$g_{ji}^\prime=\lambda_j(x)^{-1}g_{ji}(x)\lambda_i(x),\qquad x\in V_i\cap V_j.$$
\end{lemma}
\begin{corollary}
\label{birkhofffac}
The Birkhoff factorization \lemref{birkho} implies the Grothendieck theorem \theoref{grothendieck}.
\end{corollary}
\begin{proof}
Let us denote by $\mathscr D_\infty$, the set obtained by adjoining $\infty$ to $ \mathscr D$, then $ \mathscr D_\infty\cup\C=\mathbb P^1(\C)$ and $ \mathscr D_\infty\cap \C=\mathscr D$. Let
\begin{equation*}
\begin{split}
\lambda_{ \mathscr D_\infty}&\colon \mathscr D_\infty\to\rm{GL}(n,\C)\\
&z\mapsto \mathcal P\left(1/z\right)^{-1}
\end{split}
\end{equation*}
and
\begin{equation*}
\begin{split}
\lambda_{\C}&\colon\C\to\rm{GL}(n,\C)\\
&z\mapsto \mathcal E\left(z\right).
\end{split}
\end{equation*}
Given a matrix $\mathcal T(z)$, we can construct an analytic vector bundle over the Riemann sphere of rank $n$, with transition functions $\mathcal T(z)$. Similarly given a matrix $\Lambda(z)$, one can build another vector bundle over the Riemann sphere, with transition functions $\Lambda(z)$. Formula \eqref{birkhof} asserts that these two vector bundles are isomorphic by using \lemref{steenrod}. Since $\Lambda(z)$ is diagonal it gives the vector bundle $\mathcal O_{\scaleto{\mathbb P^1(\C)}{7pt}}(k_1)\oplus\ldots\oplus\mathcal O_{\scaleto{\mathbb P^1(\C)}{7pt}}(k_n)$. Hence every vector bundle over $\mathbb P^1(\C)$ is analytically isomorphic to a direct sum of line bundles of the form
$$\mathcal O_{\scaleto{\mathbb P^1(\C)}{7pt}}(k_1)\oplus\ldots\oplus\mathcal O_{\scaleto{\mathbb P^1(\C)}{7pt}}(k_n),\quad k_i\in\Z.$$
\end{proof}

\end{document}